\documentclass[10pt]{amsart}

\usepackage{graphicx}

\allowdisplaybreaks
\usepackage{enumerate,amssymb} 
\usepackage{mathrsfs} 
\usepackage{tikz} 
\usepackage{todonotes}

\newtheorem{theorem}{Theorem}[section]
\newtheorem{lemma}[theorem]{Lemma}

\theoremstyle{definition}

\newtheorem{example}[theorem]{Example}

\newtheorem{proposition}[theorem]{Proposition}
\newtheorem{corollary}[theorem]{Corollary}

\theoremstyle{remark}
\newtheorem{remark}[theorem]{Remark}

\numberwithin{equation}{section}

\DeclareMathOperator{\dv}{\textrm{div}}
\DeclareMathOperator{\grad}{grad}
\DeclareMathOperator{\Ric}{Ric}
\DeclareMathOperator{\inj}{Inj\,}
\DeclareMathOperator{\vol}{\mathrm{vol}\,}
\DeclareMathOperator{\dS}{\mathrm{dS}}

\DeclareMathOperator{\dvol}{\,\mathrm{dvol}}

\DeclareMathOperator{\dist}{\textrm{dist}\,}

\begin{document}

\title{Kuttler-Sigillito's Inequalities and Rellich-Christianson Identity}

\author[S.~M.~Berge]{Stine Marie Berge}
\address{Institute for Analysis, Leibniz University Hannover, Welfengarten 1 30167 Hannover, Germany}
\email{stine.berge@math.uni-hannover.de}


\date{}

\dedicatory{}

\commby{}

\begin{abstract}
  This article has two main objectives. The first one is to show Kuttler-Sigillito's type inequalities involving the mixed Neumann-Dirichlet, mixed Steklov-Dirichlet, and mixed Robin-Dirichlet eigenvalue problems. We will provide examples on e.g. squares and balls for the inequalities presented. Next we will show a Rellich identity for the mixed Neumann-Dirichlet problem. This identity will be used to prove a Rellich-Christianson identity for the Neumann-Dirichlet problem.
\end{abstract}

\maketitle
\section{Introduction}
In this paper we will study various eigenvalue problems with mixed conditions on the boundary. The first problem of interest is the eigenvalue problem for the Laplace operator on a bounded domain $\Omega$ on Riemannian manifolds with mixed Neumann-Dirichlet conditions on the boundary $\partial \Omega$. This problem has for example been studied in \cite{LR17, JLNP06}. In the first part of the article we are going to compare these eigenvalues with eigenvalues of the mixed Steklov-Dirichlet problem. More precisely, we consider the following problem $\Delta u = 0$ on $\Omega$ with boundary conditions $u_n=\sigma u$ on part of the boundary $F\subset \partial \Omega$ and Dirichlet conditions $u=0$ on $\partial \Omega \setminus F$. This problem has been studied in e.g.~\cite{Se21,BKPS10, HL19}. When the Dirichlet conditions are not present the problem is known as the Steklov problem, for historical context see \cite{KKKNPP14}.

Kuttler and Sigillito showed in \cite{KS68} several bounds for eigenvalue problems in $\mathbb{R}^2$. One of their results is as follows: Let $\Omega\subset \mathbb{R}^2$ be a bounded star convex domain  with $C^2$-boundary. Denote by $\mu_2$ the first non-zero eigenvalue of the Laplace eigenvalue problem with Neumann boundary conditions and let $\sigma_2$ be the first non-zero eigenvalue of the Steklov problem $\Delta u=0$ on $\Omega$ and $\sigma_2u=u_n$ on $\partial\Omega$. Then 
\begin{equation}\label{eq:intro:1}
\sigma_2\ge \frac{h_{\min} \mu_2}{2r_{\max}\sqrt{\mu_2}+2},
\end{equation}
where $h_{\min}$ and $r_{\max}$ are geometric constants depending explicitly on $\Omega$. In \cite{HS19} Hassannezhad and Siffert generalized \eqref{eq:intro:1} to Riemannian manifolds of arbitrary dimension.

One of the tools used in \cite{KS68} and later in \cite{HS19} is the Rellich identity which expresses the eigenvalue in terms of the eigenfunction and its gradient on the boundary. The simplest example is the Laplace Dirichlet eigenvalue problem $\Delta u +\lambda u= 0$ on the $C^2$-domain $\Omega\subset \mathbb{R}^n$ with $u=0$ on the boundary $\partial\Omega$. In this case $\lambda$ can be expressed as 
\begin{equation}\label{eq:2:intro}
   \lambda = \frac{1}{2}\int_{\partial \Omega}x\cdot \mathbf{n}\,u_n^2\dS,
\end{equation}
where $\dS$ is the surface measure on the boundary with normal $\mathbf{n}$. Rellich proved \eqref{eq:2:intro} in \cite{Re40}. Similar Rellich identities were shown to hold for other eigenvalue problems, e.g.\ the Neumann, clamped plate and buckling problem see \cite{Li07}. The Rellich identities in \cite{Re40, Li07} were also recently extended to eigenvalue problems on Riemannian manifolds \cite{HS19}. 

For Dirichlet eigenfunctions on the triangle one can get an interesting reformulation of the Rellich identity as shown in \cite{Ch19} by Christianson. It was further extended to polytopes in \cite{Me18}, and is known as the Rellich-Christianson. For a regular polytope $P$ this result gives 
\[\lambda = \frac{C}{2}\int_{\partial P}u_n^2\dS,\]
where $C$ is the distance from the center to the faces.

The first goal of this article is to generalize \eqref{eq:intro:1} to all eigenvalues. In particular for a star-shaped domain $\Omega \subset \mathbb{R}^2$ this means that if $\sigma_k$ and $\mu_k$ are the $k$'th eigenvalues to the Steklov problem and the Neumann problem, respectively we have
\[ \sigma_k\ge \frac{h_{\min} \mu_k}{2r_{\max}\sqrt{\mu_k}+2}.\]
More precisely, we prove an inequality similar to \eqref{eq:intro:1} for mixed Neumann-Dirichlet problem and Steklov-Dirichlet problem on sub-domains in Riemannian manifolds of arbitrary dimension, generalizing the result of \cite{HS19}. Moreover, we will show an inequality for the first mixed Steklov-Dirichlet eigenvalue, the mixed Neumann-Dirichlet, and the mixed Robin-Dirichlet eigenvalues. In the last part of the article we will show a Rellich type identity on $\mathbb{R}^n$ for the mixed Neumann-Dirichlet problem. We will use this identity to show a Rellich-Christianson identity for the mixed Neumann-Dirichlet problem on polytopes.

An outline of the article is as follows: In Section \ref{sec:Kut} we will give the Kuttler-Sigillito's type inequalities. The inequality will be obtained on Riemannian manifolds satisfying a curvature bound. In Section~\ref{sec:hadamard} we will discuss the Hadamard formula and prove it in the case of the Laplace eigenvalue problem with mixed Neumann-Dirichlet boundary conditions. Section~\ref{sec:rellich} will be devoted to proving the Rellich identities using Hadamard formulas for the mixed Neumann-Dirichlet problem. In the last two sections we will prove the Rellich-Christianson identity by using the Rellich identity on polytopes for the mixed Neumann-Dirichlet problem.

\section{Kuttler-Sigillito's Type Inequalities}\label{sec:Kut}
Let $(M,\mathbf{g})$ be a smooth $n$ dimensional Riemannian manifold. Denote by $\Ric$ the Ricci curvature of $M$, and assume that $(M,\mathbf{g})$ satisfies the curvature bound
\begin{equation}\label{curv bound}
 (n-1)\kappa |X|^2\le \mathrm{Ric}(X,X), 
\end{equation}
where $|X|^2= \mathbf{g}(X,X)$ for a vector field $X$ and $\kappa \in \mathbb{R}$.
The Laplacian with respect to $\mathbf{g}$ of a twice differentiable function $\phi$ is defined by
\[\Delta \phi =\dv(\grad \phi),\]
where $\dv$ and $\grad $ denotes the divergence and gradient with respect to $\mathbf{g}$. 

In this text we will let $r(x)$ denote the radial distance function from a chosen point $p\in M$. We will always work with domains which are contained inside a geodesic ball centered at $p$ with radius less than the injectivity radius $\inj(p)$. Some properties of the radial distance function are that $|\grad r|=1$ and $r^2$ is a smooth function on $B_{\inj(p)}(p)$.
For manifolds satisfying \eqref{curv bound} we have the following volume comparison result:
\begin{lemma}[Volume Comparison {\cite[Lemma 7.1.2]{Pe16}}]\label{C:Rauch}
   Assume that $(M,\mathbf{g})$ is an $n$-dimensional Riemannian manifold satisfying \eqref{curv bound}, and denote by \[\cot_\kappa(r)=\begin{cases}
      \coth(r\sqrt{-\kappa})/\sqrt{-\kappa},& \kappa<0\\
      1/r,& \kappa=0\\
      \cot(r\sqrt{\kappa})/\sqrt{\kappa},& \kappa>0\\
   \end{cases}.\] 
   Then for all $x\in M$ such that $r(x)<\inj(p)$ we have
\begin{equation*}
   \Delta r\le \left( n-1 \right)\cot_\kappa(r).
\end{equation*}
\end{lemma}

Denote by $\vol$ the volume density and let $\Omega\subset B_{\inj(p)}(p)$ be a domain with $C^2$-boundary. If $u,v\in H^1(\Omega)$ then by the trace theorem $v\in H^{1/2}(\partial \Omega)$ and $u_n\in H^{-1/2}(\partial\Omega)$. If additionally, we have that $\Delta u\in L^2(\Omega)$ then 
\begin{equation}\label{eq:divergence}
\int_\Omega \grad u \cdot \grad v\dvol = -\int_\Omega v\Delta u\dvol+\langle u_n,v\rangle_{\partial \Omega},
\end{equation}
where $\langle u_n,v\rangle_{\partial \Omega} $ is the paring between $H^{-1/2}(\partial \Omega)$ and $H^{1/2}(\partial \Omega)$, see \cite[Chap.~4]{Mc20}.
\begin{remark}\label{re:Lip}
   In the case when $\Omega\subset \mathbb{R}^n$ one can relax the smoothness condition of the boundary to be just Lipschitz, see \cite[Chap.~4]{Mc20}.
\end{remark}

\subsection{Steklov-Dirichlet and Neumann-Dirichlet Comparison}
Let $\Omega$ be compactly embedded into  the ball $B_{\inj(p)}(p)$, where the boundary of $\Omega$ is assumed to be $C^2$. Denote by $F$ a finite union of connected open subsets of $\partial\Omega$ where $\partial F$ is Lipschitz continuous. 
Denote by 
\[H_0^1(\Omega, \partial \Omega \setminus F)=\{u\in H^1(\Omega)\colon u=0 \text{ on }\partial\Omega\setminus F\},\]
where the boundary values are well defined by using the trace theorem.
Consider weak solutions in $H^1(\Omega)$, see \cite[Chap.~4]{Mc20} to the following eigenvalue problems:
   \textit{Steklov-Dirichlet problem}    
   \begin{equation}\label{stek}
\begin{cases}
   \Delta u(x)=0& \text{for } x\in \Omega\\
   u_n(x)=\sigma^F u(x) &\text{for }x\in  F\\
   u(x)=0 &\text{for } x\in\partial\Omega\setminus F,
\end{cases}
   \end{equation}
and \textit{Neumann-Dirichlet problem}
\begin{equation}\label{neumann}
\begin{cases}
   \Delta v(x)+\mu^F v(x)=0& \text{for } x\in \Omega\\
   v_n(x)=0 &\text{for }x\in F\\
   v(x)=0 &\text{for } x\in\partial\Omega\setminus F.
\end{cases}
\end{equation}
Notice that in the case $F=\partial \Omega$ we get the Steklov problem and the Neumann problem, respectively. Additionally, in the case that $F=\emptyset$ the Steklov-Dirichlet problem have only the trivial solution $u=0$, and the mixed Neumann-Dirichlet problem is the Dirichlet problem.

From \cite{Ag06} it follows that \[0\le \sigma_1^F\le\dots \le \sigma_n^F\to \infty.\]
For Neumann-Dirichlet eigenvalues we also have that 
\[0\le \mu_1^F\le\dots \le \mu_n^F\to \infty .\]
This can be shown by using the min-max theorem \cite[Theorem 5.15]{Bo20} together with a standard argument showing that $(-\Delta +1)^{-1}$ is a compact operator on $L^2(\Omega)$. 

The \textit{Rayleigh quotients} for the $k$'th eigenvalues are given by
\[\sigma_k^F=\inf_{\substack{V_k\subset H_0^1(\Omega,\partial\Omega\setminus F)\\ \dim(V_k)=k}}\sup_{u\in V_k \setminus \{0\}}\frac{\int_{\Omega}|\grad u|^2\dvol}{\int_{F}u^2\,\dS}\]
and
\[\mu^F_k=\inf_{\substack{V_k\subset H_0^1(\Omega,\partial\Omega\setminus F)\\ \dim(V_k)=k}}\sup_{v\in V_k\setminus \{0\} }\frac{\int_{\Omega}|\grad v|^2\dvol}{\int_{\Omega}v^2\dvol}.\]
The Rayleigh quotient for the mixed Steklov-Dirichlet can be found in \cite{Ag06}, while for the Neumann-Dirichlet problem the Rayleigh quotient follows from the min-max theorem \cite[Theorem 5.15]{Bo20}. Notice that the $k$'th Neumann-Dirichlet eigenvalue is bounded below by the $k$'th Neumann eigenvalue, and above by the $k$'th Dirichlet eigenvalue by using the Rayleigh quotient together with \[H^1(\Omega, \partial \Omega)\subset H^1(\Omega, \partial \Omega\setminus F)\subset H^1(\Omega).\]

Our first main result is the following generalization of \eqref{eq:intro:1}.
\begin{theorem}\label{thm:kut}
   Assume that $(M,\mathbf{g})$ satisfies \eqref{curv bound} and let $\Omega\subset M$ be a bounded $C^2$ domain contained inside $B_{\inj(p)}(p)$ for a fixed point $p\in M$. Denote by $r$ the radial distance function with respect to $p$ and let $\mu_k$ and $\sigma_k$ be as in \eqref{neumann} and \eqref{stek}, respectively. Then 
   \begin{equation}\label{eq:eigenvalue ineq}
   \sigma_k^F\ge \frac{h_{\min}\mu_k^F}{2r_{\max}\sqrt{\mu_k^F}+C_0},
   \end{equation}
   where 
   \[r_{\max}=\sup_{x\in \Omega} r(x),\qquad h_{\min}=\inf_{x\in F}\mathbf{n}\cdot r\grad r,\]
   and \[C_0=1+(n-1)\sup_{x\in\Omega}r(x)\cot_\kappa(r(x)).\]
\end{theorem}
\begin{remark}\hfill
   \begin{itemize}
      \item If $\kappa\ge 0$ we have that $C_0 = n$ since $r\cot_\kappa(r)$ is non-increasing in $r$. Additionally, for $\kappa<0$ the function $r\cot_\kappa(r)$ is increasing, hence the maximum is obtained on the boundary.
      \item Inequality \eqref{eq:eigenvalue ineq} is non-trivial if and only if $h_{\min}> 0$. In the case that $M=\mathbb{R}^n$ and $\mathbf{0}\in \Omega$ the condition $h_{\min}>0$ means that $x$ cannot be tangent to the part of the boundary $F$. In the case that $F=\partial \Omega$ this condition implies that $\Omega$ is strictly star convex with respect to the point $\mathbf{0}$. We define a domain to be strictly star convex with respect to a point $p$ if any straight line from $p$ to the boundary intersect the boundary only at the end point.
      
      It should also be noted that all convex sets containing zero satisfy $h_{\min}>0$.
      \item The Weyl laws on $\Omega \subset \mathbb{R}^n$ in the case of the Steklov problem \cite[Sec.~5.2.1]{GP17} and Neumann problem \cite[p.~9]{Ch84} are 
      \begin{equation*}
      \lim_{k\to \infty}\frac{\sigma_k}{k^{\frac{1}{n-1}}}= 2\pi \left(\frac{1}{\omega_{n-1}\vol_{n-1}(\partial \Omega)}\right)^{\frac{1}{n-1}}
      \end{equation*}
      and 
      \begin{equation}\label{eq:weyl:Neumann}
      \lim_{k\to \infty}\frac{\mu_k}{k^{2/n}}= (2\pi)^{2} \left(\frac{2}{\omega_{n}\vol_{n}( \Omega)}\right)^{\frac{1}{n}},
      \end{equation}
      where $\omega_m$ denotes the volume of the $m$-dimensional unit ball.
      In this case we have that the right hand-side of \eqref{eq:eigenvalue ineq} behaves like $C_0(\Omega)k^{\frac{1}{n-1}}$ and the left hand-side behaves like $C_1(\Omega)k^{\frac{1}{n}}$. This shows that the result is trivially true when $k$ goes to infinity.
   \end{itemize}
\end{remark}
\begin{proof}
   When $k=1$ and $F=\partial \Omega$ we have that $\sigma_1^F=\mu_1^F=0$ and the inequality trivially holds. Hence we will assume that either $k\ge 2$ or $\overline{F}\not = \partial \Omega$.

   Let $\sigma_m^F$ be the $m$'th eigenvalue of the Steklov-Dirichlet problem with corresponding eigenfunction $u_m$, and denote by 
   \[V_k=\textrm{span}\{u_1,\dots ,u_k\}\subset H^1_0(\Omega, \partial \Omega \setminus F).\]
   Then for any $v\in V_k\setminus \{0\}$ we have that 
   \[\frac{\int_{\Omega}|\grad v|^2\dvol}{\int_Fv^2\dS}\le\sigma_k^F.\]
   Choose $v\in V_k\setminus \{0\}$ to satisfy
   \[\frac{\int_{\Omega}|\grad v|^2\dvol}{\int_{\Omega}v^2\dvol}=\sup_{u\in V_k\setminus \{0\}}\frac{\int_{\Omega}|\grad u|^2\dvol}{\int_{\Omega}u^2\dvol}.\]
   By the definition of the eigenvalues of the mixed Neumann-Dirichlet problem we get that 
   \begin{equation}\label{eq:neu}
   \mu_k^F\int_{\Omega}v^2\dvol\le \int_{\Omega}|\grad v|^2\dvol.
   \end{equation}
   Since we are inside the ball $B_{\inj(p)}(p)$ we have that $r^2$ is a smooth function.
   Using the integration by parts \eqref{eq:divergence} and Lemma \ref{C:Rauch} we obtain
   \begin{align}\label{eq:normal}
   2\int_Fv^2\mathbf{n}\cdot r\grad r\dS&=\int_Fv^2\mathbf{n}\cdot\grad (r^2)\dS\\
   &=\int_\Omega \grad v^2\cdot \grad r^2 \dvol+\int_\Omega v^2 \Delta(r^2)\dvol\nonumber\\
   &=2\int_\Omega \grad v^2\cdot r \grad r \dvol+\int_\Omega v^2( 2+ 2r\Delta r)\dvol\nonumber\\
   &\le 2\int_\Omega  \grad v^2\cdot r\grad r \dvol+2C_0\int_\Omega v^2\dvol.\nonumber
   \end{align}
   The Cauchy-Schwartz inequality together with \eqref{eq:neu} gives 
   \begin{equation}\label{eq:CS}
   \int_\Omega v\grad v\cdot r\grad r\dvol\le \frac{r_{\max}}{\sqrt{\mu_k^F}}\int_\Omega|\grad v|^2\dvol.
   \end{equation}
Using \eqref{eq:neu}, \eqref{eq:normal}, and \eqref{eq:CS} gives 
   \[\int_Fv^2\mathbf{n}\cdot r\grad r\dS\le \left(2r_{\max}(\mu_k^F)^{-1/2}+C_0/\mu_k^F\right)\int_\Omega |\grad v|^2\dvol.\]
   Simplifying the expression
   \[h_{\min}\int_Fv^2\dS\le \left(2r_{\max}(\mu_k^F)^{-1/2}+C_0/\mu_k^F\right)\sigma_k^F\int_Fv^2\,\dS\]
   gives the result.
\end{proof}

In the special case when $\Omega=B_{R}(p)$ we get the following corollary.
\begin{corollary}\label{Cor:ball}
Let $(M,\mathbf{g})$ be an $n$-dimensional Riemannian manifold and let $\Omega = B_{R}(p)\subset M$ where $R<\inj(p)$ and $F\subset \partial B_{R}(p)$. Then 
\[\sigma_k^F\ge \frac{R\mu^F_k}{2R\sqrt{\mu^F_k}+C_0},\]
 where
\[C_0=
\begin{cases}
 n & \kappa\ge 0\\  
1+(n-1)R\cot_\kappa(R) & \kappa <0
   \end{cases}.
\]
\end{corollary}
\begin{example}
  Let $n=2$ and $(M,\mathbf{g})$ be the hyperbolic space with curvature $-1$. In this case we have that the Steklov eigenvalues on the ball $B_R(p)$ are given by $\sigma_k=\frac{(\lfloor k/2\rfloor)^2}{\sinh(R)}$, where $\lfloor\cdot \rfloor$ denotes the floor function. Note that by the formula we have that each eigenvalue for $k>1$ have multiplicity two. The corresponding eigenfunctions written in geodesic polar coordinates are given by 
  \[u_{2m}(r,\theta)=\tanh(r/2)^{m^2}\sin( m\theta)\quad \text{and}\quad u_{2m+1}(r,\theta)=\tanh(r/2)^{m^2}\cos( m\theta),\]
  where $k$ is either $2m$ or $2m+1$.
  Using Corollary \ref{Cor:ball}, the Neumann eigenvalue $\mu_k(R)$ on the ball $B_R(p)$ satisfies 
\[\frac{(\lfloor k/2\rfloor)^2}{\sinh(R)}\ge \frac{R\mu_k(R)}{2R\sqrt{\mu_k(R)}+1+R\coth(R)}.\]
Taking the limit as $R\to \infty$ gives that 
\[\lim_{R\to \infty}\mu_k(R)= 0.\]
When $k=1$ this is already known, see \cite[p.~46 Thm.~5]{Ch84}.
\end{example}

By using Remark \ref{re:Lip} we have that for bounded domains in $\mathbb{R}^n$ with Lipschitz boundary the proof of Theorem \ref{thm:kut} still holds. Hence we can explore the following examples.
\begin{example}\label{ex:half circle}
   Let $J_l(r)$ denote the $l$'th \textit{Bessel function} for $l\in \mathbb{N}\setminus \{0\}$, which satisfies the differential equation \[r^2J_l''(r)+rJ_l'(r)+(r^2-l^2)J_l(r)=0.\]
   Consider the half-ball
   \[D_+ = \{ (x,y)\in \mathbb{R}^2 \colon x^2+y^2< 1,\, y>0\}.\]
   On this domain we will consider the problems
   \[
\begin{cases}
   \Delta u(x,y)=0& \text{on }  D_+\\
   u_n(x,y)=\sigma u &\text{on } F=\{ (x,y)\in \mathbb{R}^2 \colon x^2+y^2= 1,\, y>0\}\\
   u(x,y)=0 &\text{on } \{(x,0)\in \mathbb{R}^2\colon |x|\le 1\},
\end{cases}
\]
and 
\[
\begin{cases}
   \Delta v(x,y)+\mu v(x,y)=0& \text{on } D_+\\
   v_n(x,y)=0 &\text{on } F=\{ (x,y)\in \mathbb{R}^2 \colon x^2+y^2= 1,\, y>0\}\\
   v(x,y)=0 &\text{on }  \{(x,0)\in \mathbb{R}^2\colon |x|\le 1\}.
\end{cases}
\]
 The solutions to the Neumann-Dirichlet problem are given by \[v(r,\theta) =J_l(j_{l,m}' r)\sin(l\theta)\qquad l\in \mathbb{N},\] 
 where $0<j'_{l,m}$ is the $m$'th zero of $J_l'$ and $\mu=(j'_{l,m})^2$. The Steklov-Dirichlet problem have the solutions \[u(r,\theta) = r^k\sin(k\theta)\qquad k\in \mathbb{N}\]
 with corresponding eigenvalues $\sigma_k=k>0$. If we order the zeroes $j'_{l,m}$ we have that 
   \[\sigma_k=k\ge \frac{j_{l_k,m_k}'^2}{2(j_{l_k,m_k}'+1)}=\frac{\mu_k}{2\sqrt{\mu_k}+2}.\]

   We can extend the Neumann-Dirichlet eigenfunctions above to Neumann eigenfunction on the ball by using a reflection. Recall that for the Neumann problem all the eigenspaces corresponding to non-constant eigenfunctions have dimension $2$. The solutions given by the extension from the Neumann-Dirichlet problem gives us one of the two linearly independent solutions of the eigenspace of the Neumann problem. The Weyl law presented in \eqref{eq:weyl:Neumann} shows that $\mu_k\asymp k$.
\end{example}

\begin{example}\label{ex:stekolov:square}
   Let $\Omega=[0,1]\times [0,1]$ and let $F=\{0,1\}\times (0,1)$. We pose the problem 
   \begin{equation}\label{eq:mixed_square}
   \begin{cases}
      \Delta v(x,y)+\mu v(x,y)=0& \text{on } \Omega\\
      v_n(x,y)=0 &\text{on } F\\
      v(x,y)=0 &\text{on } [0,1]\times \{0,1\}.
   \end{cases}
   \end{equation}
      Then we have the solutions
      \[v(x,y)=\cos\left(\pi mx\right)\sin\left(\pi ny\right)\]
      with corresponding eigenvalues $\mu=(\pi m)^2+(\pi n)^2$, where $m\in \mathbb{N}\cup \{0\}$ and $n\in \mathbb{N}\setminus \{0\}$.
   
   The Steklov-Dirichlet problem on the same domain is given by 
   \[
   \begin{cases}
      \Delta u(x,y)=0& \text{on } \Omega\\
      u_n(x,y)=\sigma u(x,y) &\text{on } F\\
      u(x,y)=0 &\text{on } [0,1]\times \{0,1\}.
   \end{cases}
   \]
   Computing the eigenvalues gives
   \[
   \sigma=\pi n\left( \frac{\cosh\left(\frac{\pi n}{2}\right)}{\sinh\left(\frac{\pi n}{2}\right)}\right)^{\pm 1}
   \]
   with corresponding eigenfunctions
   \begin{align*}
   u(x,y)&=\Bigg(\cosh\left(\pi nx\right)-\frac{\sigma}{n\pi}\sinh\left(\pi nx\right)\Bigg)\sin\left(\pi ny\right),
   \end{align*} 
   for all $n\in \mathbb{N}\setminus\{0\}$.
   Using that \[\pi (n-1)\coth\left(\frac{\pi(n-1)}{2}\right)<\pi n\tanh\left(\frac{\pi n}{2}\right)\]
    for all $n\in \mathbb{N}\setminus \{0,1\}$, we can find the order of the eigenvalues. Hence we get 
   \[
   \sigma_{2k}=\pi k\coth\left(\frac{\pi k}{2}\right)
   \]
   with corresponding eigenfunctions
   \begin{align*}
   u_{2k}(x,y)&=\Bigg(\cosh\left(\pi kx\right)-\coth\left(\frac{\pi k}{2}\right)\sinh\left(\pi kx\right)\Bigg)\sin\left(\pi ky\right)
   \end{align*} 
   and 
   \[
   \sigma_{2k-1}=\pi k\tanh\left(\frac{\pi k}{2}\right)
   \]
   with corresponding eigenfunctions
   \begin{align*}
   u_{2k-1}(x,y)&=\Bigg(\cosh\left(\pi kx\right)-\tanh\left(\frac{\pi k}{2}\right)\sinh\left(\pi kx\right)\Bigg)\sin\left(\pi ky\right).
   \end{align*} 
   In this case we have that $r_{\max} = \frac{1}{\sqrt{2}}$, $h_{\min} = 1/2$, and $C_0=2$. 
   Using Theorem~\ref{thm:kut} we obtain
   \[\sigma_{2k}=\pi k\coth\left(\frac{\pi k}{2}\right)\ge \frac{\pi^2(m_{2k}^2+n_{2k}^2)}{2\sqrt{2}\pi\sqrt{m_{2k}^2+n_{2k}^2}+4}.\]
   and 
   \[\sigma_{2k-1}=\pi k\tanh\left(\frac{\pi k}{2}\right)\ge \frac{\pi^2(m_{2k-1}^2+n_{2k-1}^2)}{2\sqrt{2}\pi\sqrt{m_{2k-1}^2+n_{2k-1}^2}+4},\]
   where $\mu_j=\pi^2(m_j^2+n_j^2)$.
\end{example}

\subsection{Robin-Dirichlet Comparison}
In this section we will assume that $\partial \Omega \setminus F$ has non-empty interior in $\partial\Omega$. The \textit{Robin-Dirichlet} problem for any $\alpha>0$ is defined by
\[
\begin{cases}
   \Delta w(x)+\lambda^\alpha w(x)=0& \text{for } x\in \Omega\\
   w_n(x) + \alpha w=0 &\text{on } F\\
   w(x)=0 &\text{on } \partial\Omega\setminus F.
\end{cases}
\]
Note that we have suppressed the domain $F$ in the Robin-Dirichlet eigenvalues $\lambda^\alpha$ for readability. 
By the min-max theorem \cite[Theorem 5.15]{Bo20} we have that the Rayleigh quotient is given by 
\[\lambda^\alpha_k=\inf_{\substack{V_k\subset H_0^1(\Omega,\partial\Omega\setminus F)\\ \dim(V_k)=k}}\sup_{w\in V_k \setminus \{0\}}\frac{\int_{\Omega}|\grad w|^2\dvol+\alpha\int_{F}w^2\,\dS}{\int_{\Omega}w^2\dvol}.\] 
The eigenvalue of the Robin-Dirichlet problem is always between the Dirichlet eigenvalues and the Neumann-Dirichlet eigenvalues with the Dirichlet eigenvalue being the limit as $\alpha \to \infty$.
Again we have only point spectrum since $(-\Delta + 1)^{-1}$ is a compact operator on $L^2(\Omega)$.

\begin{theorem}\label{thm:kutlertwo}
   Assume that $\bar{F}\not = \partial \Omega$, then eigenvalues satisfy
   \[\frac{\lambda^\alpha_k - \mu_k^F}{\alpha}\le \frac{\mu_k^F}{\sigma_1^F}.\]
\end{theorem}
\begin{proof}
   For all $u\in H_0^1(\Omega, \partial\Omega\setminus F)$ we have that 
   \[\sigma_1^F\int_{F}u^2\,\dS\le \int_{\Omega}|\grad u|^2\dvol.\]
   Denote by $v_m$ the eigenfunctions counting multiplicity of the Neumann-Dirichlet problem with eigenvalue $\mu_m$. Let $V_k$ consists of the span of $v_1, \dots, v_k$. 
   By the Rayleigh quotient the eigenvalue $\lambda_k^\alpha$ satisfies
   \[\lambda^\alpha_k \le \sup_{v\in V_k\setminus\{0\} }\frac{\int_{\Omega}|\grad v|^2\dvol+\alpha\int_{F}v^2\,\dS}{\int_{\Omega}v^2\dvol} \le \mu^F_k+ \alpha\frac{\mu^F_k}{\sigma_1^F}.\qedhere\]
\end{proof}

\begin{remark}\label{rem:derivative}
Taking the limit when $\alpha \to 0$ and assuming that $\lambda^\alpha_k$ is differentiable at $0$, we have that 
   \[\left.\frac{d\lambda^\alpha_k }{d\alpha}\right|_{\alpha = 0}\le \frac{\mu_k^F}{\sigma_1^F}.\] 
\end{remark}
\begin{example}
   In this example we will continue using some of the notation from Example \ref{ex:half circle}.
   Consider the equation 
   \[\alpha J_l(r)+rJ_l'(r)=0.\]
   Denote by $j_{l,m}^\alpha$ the $m$'th zero of the equation.
   Posing the problem
   \[
\begin{cases}
   \Delta w(x, y)+\lambda^\alpha w(x,y)=0& \text{for } (x,y)\in D_+\\
   w_n(x,y)+\alpha w(x,y)=0 &\text{on } F= \{ (x,y)\in \mathbb{R}^2 \colon x^2+y^2= 1, y>0\}\\
   w(x,y)=0 &\text{on } \{(x,0)\in \mathbb{R}^2\colon |x|\le 1\},
\end{cases}
\]
we get the solution $w(r,\theta) = J_l(j_{l,m}^\alpha r) \sin(l\theta)$, with eigenvalue $\lambda^\alpha =  (j_{l,m}^\alpha)^2$. Using Theorem \ref{thm:kutlertwo} we obtain 
\[\frac{(j_{l_k,m_k}^\alpha)^2-(j_{l_k,m_k}')^2}{\alpha}\le (j_{l_k,m_k}')^2=\mu_k.\]

Taking the derivative of the equation 
   \[\alpha J_l(j_{l,m}^\alpha)+j_{l,m}^\alpha J_l'(j_{l,m}^\alpha)=0,\]
   with respect to $\alpha$ gives
   \[J_l(j_{l,m}^\alpha)+\alpha \frac{dj_{l,m}^\alpha}{d\alpha} J_l'(j_{l,m}^\alpha)+\frac{dj_{l,m}^\alpha}{d\alpha} J'_l(j_{l,m}^\alpha )+j_{l,m}^\alpha\frac{dj_{l,m}^\alpha}{d\alpha}J_l''(j_{l,m}^\alpha)=0.\]
   Taking the limit when $\alpha \to 0$ shows that 
   \[J_l(j_{l,m}')+j_{l,m}'\left.\frac{dj_{l,m}^\alpha}{d\alpha}\right|_{\alpha=0}J_l''(j_{l,m}')=J_l(j_{l,m}')+j_{l,m}'\left.\frac{dj_{l,m}^\alpha}{d\alpha}\right|_{\alpha=0}\frac{l^2-(j'_{l,m})^2}{(j'_{l,m})^2}J_l(j_{l,m}')=0,\]
   where we have used that $J_l'(j_{l,m}')=0$.
   Simplifying the derivative we get
   \begin{align*}
   \left.\frac{d\lambda^\alpha}{d\alpha}\right|_{\alpha=0}&=\left.\frac{d(j_{l_k,m_k}^\alpha)^2}{d\alpha}\right|_{\alpha=0}\\
   &=2j'_{l_k,m_k}\left.\frac{dj_{l_k,m_k}^\alpha}{d\alpha}\right|_{\alpha=0}\\
   &=2\frac{(j_{l_k,m_k}')^2}{(j_{l_k,m_k}')^2-l_k^2}.
   \end{align*}
   Using Remark \ref{rem:derivative} and Example \ref{ex:half circle} we get 
   \begin{equation}
   \left.\frac{d\lambda^\alpha}{d\alpha}\right|_{\alpha=0}=2\frac{(j_{l_k,m_k}')^2}{(j_{l_k,m_k}')^2-l_k^2}\le \frac{\mu_k}{\sigma_1}=(j_{l_k,m_k}')^2.
   \end{equation}
   Simplifying this expression we get $j'_{l,m}\ge \sqrt{l^2+2}$. This inequality follows from the stronger inequality $j'_{l,m}>j'_{l,1}> \sqrt{l(l+2)}$, which can be found in \cite[p.~486]{Wa95}. In particular, since we also have that $j'_{l,1}\le l+Cl^{1/3}$, see e.g.\ \cite{EL97}, we get 
   \[\lim_{l\to \infty} \frac{j'_{l,1}}{\sqrt{l^2+2}}=1.\]
   This makes the inequality sharp in the limit.
\end{example}

\section{Hadamard Formula for the Mixed Neumann-Dirichlet Problem}\label{sec:hadamard}
The classical Hadamard formula on the Dirichlet Laplace problem is as follows: Let $u_t$ be a Dirichlet eigenfunction with the $\lambda_t$ eigenvalue on the time dependent domain $\Omega_t\subset \mathbb{R}^n$ with the speed $V$. Then we have that the Hadamard formula gives that 
\[\frac{d\lambda}{dt}=-\int_{\partial \Omega_t} V\cdot \mathbf{n} u^2\dS.\]
In the case of the Riemannian manifold can be found in e.g.\ \cite{EI07}.
To be more precise, let $V$ be a (possibly time dependent) vector field and let $\Omega_t$ be the deformation of the domain $\Omega_0$ with respect to the corresponding one-parameter family of diffeomorphisms $\phi$ with speed $V$. Assume that $\Omega_t$ is pre-compact with smooth boundary $\partial\Omega_t$, and $F_t = \phi_t(F)$. We will use the superscript $t$ to denote the time-dependence.

Let $v^t \in H^2(\Omega)\cap C^1(\overline{\Omega})$ be a solution to the mixed Neumann-Dirichlet problem 
 \[\begin{cases}
      \Delta v^t(x) +\mu^t v^t(x) =0 & \text{in } \Omega_t\\
      v^t_n(x) = 0 & \text{on }F_t\\
      v^t(x) = 0 & \text{on }\partial\Omega_t\setminus F_t.
    \end{cases}\]
    We will prove the Hadamard formula
    \begin{equation}\label{eq:hademard_ND}
    \partial_t\mu^t=\int_{F_t}\left(|\grad_{\partial \Omega_t}v^t|^2-\mu^t(v^t)^2\right)\dS-\int_{\partial \Omega_t\setminus F_t}\mathbf{n}\cdot V\,(v^t_n)^2 \dS.
    \end{equation}
    Formula \eqref{eq:hademard_ND} is written down in \cite[Eq.~(11)]{Gr10}. We will assume the eigenvalues are simple and that $\mu^t$ and $v^t$ is differentiable in $t$, where $\frac{dv^t}{dt}\in C^1(\overline{\Omega_t})$. For the case of non-simple eigenvalues see \cite{Gr10}.
When taking the derivative of integrals it will be useful to have the following lemma in mind:
\begin{lemma}[{\cite[p. 138]{Fr12}}]
   Let $V$ be a possibly time dependent vector field, and let $\Omega_t$ be the variation of $\Omega_0$ with respect to $V$. Assume furthermore that $\Omega_t$ is pre-compact with smooth boundary $\partial\Omega_t$. Then given a smooth function $h^t\in C^1(\Omega_t)$ which is differentiable in $t$ we have that 
   \[\frac{d}{dt}\int_{\Omega_t}h^t\dvol=\int_{\Omega_t}\partial_t h^t\dvol+\int_{\partial \Omega_t}\mathbf{n}\cdot V \,h^t\, \dS.\]
\end{lemma}
\begin{proof}[Proof of \eqref{eq:hademard_ND}]
Note that since 
\[\int_{\Omega_t} (v^t)^2\vol = 1\]
 with the boundary conditions we get that 
\begin{equation*}
2\int_{\Omega_t}v^t\partial_t v^t \dvol+\int_{F_t}V\cdot \mathbf{n}\,(v^t)^2\dS=0.
\end{equation*}
The (normalized) Rayleigh quotient gives us that \[\mu^t = \int_{\Omega_t}|\grad v^t|^2\dvol.\]
Taking the derivative of $\mu^t$ with respect to $t$ and using the divergence theorem gives 
\begin{align*}
\partial_t\mu^t&=2\int_{\Omega_t}\grad \partial_t v^t\cdot \grad v^t\dvol +\int_{\partial \Omega_t}\mathbf{n}\cdot V\,|\grad v^t|^2\dvol\\
&=2\int_{\partial \Omega_t\setminus F_t}v^t_n\partial_t v_t\dS-2\int_{\Omega_t}(\Delta v^t)\partial_tv^t\dvol+\int_{\partial\Omega_t}\mathbf{n}\cdot V\,|\grad v^t|^2\dS\\
&=2\int_{\partial \Omega_t\setminus F_t}v^t_n\partial_t v^t\dS-\mu^t\int_{F_t}V\cdot \mathbf{n} (v^t)^2\dS+\int_{\partial\Omega_t}\mathbf{n}\cdot V\,|\grad v^t|^2\dS.
\end{align*}

For $x\in \partial\Omega_t\setminus F_t$ using the chain rule gives
\[0= \frac{d}{dt}(v^t(\phi_t(x)))= \frac{dv^t}{dt}(\phi_t(x))+V\cdot \grad v^t=\frac{dv^t}{dt}(\phi_t(x))+V\cdot \mathbf{n} v^t_n,\]
where we have used that gradient in the tangent direction to $\partial \Omega_t\setminus F_t$ of $v$ is $0$.
This gives that
\begin{align*}
\partial_t\mu^t&=-2\int_{\partial \Omega_t\setminus F_t}V\cdot \mathbf{n}\,(v_n^t)^2\dS-\mu^t\int_{F_t}V\cdot \mathbf{n} (v^t)^2\dS+\int_{\partial\Omega_t}\mathbf{n}\cdot V\,|\grad v^t|^2\dS\\
&=\int_{F_t}\left(|\grad_{\partial \Omega_t}v^t|^2-\mu^t(v^t)^2\right)\dS-\int_{\partial \Omega_t\setminus F_t}\mathbf{n}\cdot V\,(v^t_n)^2 \dS.&\qedhere
\end{align*}
\end{proof}
\section{Rellich Identity}\label{sec:rellich}
Applying the Hadamard formulas in the case that $V$ is the vector field $x\mapsto x$ we can obtain Rellich type identities.
The time independent version of the problems in the previous section are given as follows:
Denote by $\Omega\subset \mathbb{R}^n$ a domain with piece-wise smooth boundary $\partial \Omega$. We will assume that the solutions of the mixed Neumann-Dirichlet are normalized in $L^2(\Omega)$. 
We are going to show that the eigenvalue $\mu^F$ with solution $v\in H^2(\Omega)\cap C^1(\overline{\Omega})$ of the mixed Neumann-Dirichlet problem \eqref{neumann} satisfies
\[\mu^F\left(\int_F \mathbf{n}\cdot x\, v^2\dS-2\right)= \int_{F}\mathbf{n}\cdot x\, |\grad_{\partial \Omega}v|^2\dS-\int_{\partial\Omega\setminus F}\mathbf{n}\cdot x\, v_n^2\dS.\]

\begin{remark}
In \cite[Theorem 1.5]{HS19} the authors show the Rellich identity by using the divergence theorem. This means that when we have solutions on a Lipschitz domain and a solution in $v\in H^2(\Omega)$, where $|\grad v|^2$ on the boundary is interpreted by using the trace theorem.
\end{remark}

Denote by $\Omega\subset \mathbb{R}^n$ a smooth pre-compact domain. The idea of going from the Hadamard formula to the corresponding Rellich type identity is to evolve the domain by using the vector field $x\mapsto x$, which by abuse of notation we are going denote to $x$. We will denote the perturbation of the domain by $\Omega_t$, which is explicitly given by $e^t\Omega$. Since the domain does not change shape, only size, we can find out how the eigenvalues explicitly change. Hence setting $t = 0$ in the Hadamard formula will give us an expression for the eigenvalues using the boundary data.

Let $v$ be a solution to
\begin{equation*}
   \begin{cases}
      \Delta v(x)+\mu^F v(x) &\text{in } \Omega\\
      v_n(x)=0&\text{on }F\\
      v(x)=0&\text{on }\partial \Omega\setminus F
   \end{cases}.
\end{equation*}
Then we have that $v^t(x)=e^{-tn/2}v(e^{-t}x)$ and $\mu^t=e^{-2t}\mu^F$ solves 
\begin{equation*}
   \begin{cases}
      \Delta v^t(x)+\mu^t v^t(x)=0 &\text{in } \Omega_{t}\\
      v^t_n(x)=0&\text{on }F_{t}\\
      v^t(x)=0&\text{on }\partial \Omega \setminus F_{t}
   \end{cases},
\end{equation*}
and additionally satisfy $\int_{\Omega_t}(v^t)^2\dvol=1$. Using the Hadamard formula we have that 
\begin{equation}\label{eq:derivative}
\partial_t\mu^t=-2e^{-2t}\mu^F=\int_{e^tF}\mathbf{n}\cdot x\left(\left|\grad_{\partial \Omega_t}v^t\right|^2-\mu^t (v^t)^2\right)\, \dS-\int_{e^t\partial \Omega \setminus F}\mathbf{n}\cdot x\, (v_n^t)^2\dS.
\end{equation}
Setting $t=0$ and collecting the $\mu^F$-terms on the left-hand side we have 
    \[\mu^F\left(\int_{F}\mathbf{n}\cdot x \,v^2\,\dS-2\right)= \int_{\partial \Omega}\mathbf{n}\cdot x\,|\grad_{\partial \Omega} v|^2\, \dS-\int_{\partial \Omega \setminus F}\mathbf{n}\cdot x\, v_n^2\dS.\]
\begin{remark}
\begin{itemize}
    \item  For the general mixed Neumann-Dirichlet problem we know that in certain cases the $\int_F x\cdot \mathbf{n} v^2\dS=2$. For an example of when this happens see Example \ref{ex:square}.
    \item Note that in the case that $F=\partial \Omega$ and $\Omega$ is star-convex with respect to the origin, we have that $ \mathbf{n}\cdot x\ge 0$ with some open part of the boundary having $\mathbf{n}\cdot x> 0$. Hence assuming that $\mu \not = 0$ and using \eqref{eq:derivative} we get  \[\left(\int_{\partial\Omega}\mathbf{n}\cdot x v^2\,\dS -2\right)\mu=\int_{\partial\Omega}\mathbf{n}\cdot x\left|\grad_{\partial\Omega}v\right|^2\, \dS>0.\] In particular we have that $\int_{\partial\Omega}\mathbf{n}\cdot x v^2\,\dS -2>0,$
meaning that the denominator is not identically $0$.
\item In the case that $F=\emptyset$, we get that 
\[-2\mu^F= -\int_{\partial\Omega\setminus F}\mathbf{n}\cdot x\, v_n^2\dS,\]
which simplifies to 
\[\mu^F= \frac{1}{2}\int_{\partial\Omega\setminus F}\mathbf{n}\cdot x\, v_n^2\dS.\]
\end{itemize}
\end{remark}

\section{Rellich-Christianson Type Identity}\label{rellich-christanson}
In this section we will look on the Rellich identities applied to polytopes. It should be noted that if the solution is regular enough, the Rellich identities also work on Lipschitz domains, due to an alternative proof using the divergence theorem, see \cite{HS19,Li07}, and Section~\ref{sec:rellich}. For a good reference on the mixed Neumann-Dirichlet problem on Lipschitz domains see \cite{LR17}.

A polytope is by definition a domain restricted by a finite number of hypersurfaces in $\mathbb{R}^n$ with finite volume. We will let $P$ denote a polytope with faces $F_i$, $i = 1,\dots , d$. The hypersurfaces corresponding to $F_i$ will be denoted $H_i$. Let $p$ be an arbitrary point in $\mathbb{R}^n$. Then we define the signed distance $\dist(p)_i$ from $p$ to $H_i$ to be minus the distance from $p$ to $H_i$ in the case that $p$ and $P$ are on the same side of $H_i$, and the distance otherwise.

In the rest of this section we will assume that $p$ is at the origin, which we will denote by $\mathbf{0}$. When this is the case we have that if $x\in F_i$ then $\mathbf{n}\cdot x=\dist(\mathbf{0})_i$. In the next couple of subsections we will use this fact to simplify the Rellich identity. The Rellich-Christianson identity for the mixed Neumann-Dirichlet generalizes the Rellich-Christianson identity for the Dirichlet problem found in \cite{Me18}. We are going to present a similar proof to the one found in \cite{Me18}.
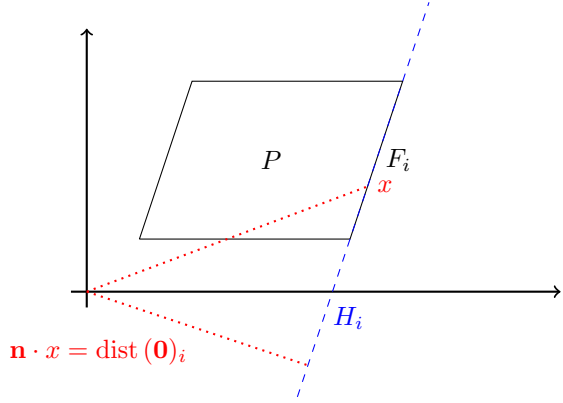
\begin{center}
\begin{tikzpicture}[scale = 0.7]
   \draw[thick,->] (-0.3,0) -- (9,0);
   \draw[thick,->] (0,-0.3) -- (0,5);
   \node at (3.5,2.5) {$P$};
   \draw (1,1) -- (2,4) -- (6,4) -- (5.5,2.5) node[anchor = west]{$F_i$} -- (5,1) -- (1,1);
   \draw[red,dotted,thick] (0,0) -- (5.3333,2) node[anchor = west]{$x$};
   \draw[dashed,blue] (4,-2) -- (4.5,-0.5) node[anchor = west]{$H_i$} -- (6.5,5.5);
   \draw[red,dotted,thick] (0,0)-- (2.1,-0.7) node[anchor = north east]{$\mathbf{n}\cdot x=\dist(\mathbf{0})_i$} -- (4.2,-1.4);
\end{tikzpicture}
\end{center}

We are going to denote by $F_i^n=F_i$ a face where we have posed the Neumann boundary conditions and $F_i^d=F_{i+j}$ the faces where we have posed Dirichlet conditions. In this case the mixed Neumann-Dirichlet problem becomes
\begin{equation}\label{eq:mixed_poly}
   \begin{cases}
      \Delta v(x)+\mu v(x)=0 &\text{in } P\\
      v_n(x)=0&\text{on }F^n_1\cup F^n_2\cup \dots\cup F^n_j\\
      v(x)=0&\text{on }F^d_1\cup F^d_2\cup \dots\cup F^d_k
   \end{cases}.
\end{equation}
To get the Rellich identity we will need that $v\in H^2(P)$. 
\begin{proposition}\label{prop:Chrisian}
Let $v\in H^2(P)$ be a solution to \eqref{eq:mixed_poly} which is normalized in $L^2(P)$.
In this case we have that $\mu$ and $v$ satisfy 
\begin{multline}\label{eq:chris:1}
\mu\left(\sum_{i=1}^j\dist(p)_i\int_{F_i^n} v^2\, \dS-2\right)=\sum_{i=1}^j\dist(p)_i\int_{F_i^n}\left|\grad_{\partial P}v\right|^2\, \dS\\
-\sum_{i=1}^k\dist(p)_{i+j} \int_{F_i^d} v_n^2\dS.
\end{multline}
\end{proposition}
\begin{remark}
\begin{itemize}\hfill
  \item  In general, one can not assume that solutions are in $H^2(\Omega)$ when $\Omega$ is Lipschitz domain. An example of this is the problem on the half disk $D_+$ given by
\begin{equation*}
    \begin{cases}
    \Delta v+\mu v=0& \text{on }D_+ \\
    v_n=0 & \text{on }\{(r,\pi)\colon  r<1\}\\
    v=0 & \text{on }\{(r,0)\colon r<1\}\cup\{(1,\theta)\colon 0<\theta<\pi\}.
    \end{cases}
\end{equation*}
A solution to this problem is given by
\[v(r,\theta)=J_{1/2}(j_{1,1/2}r)\sin(\theta/2),\]
where $j_{1,1/2}$ is the first zero of the Bessel function $J_{1/2}$.
Using the Taylor series \[J_{1/2}(x)=\sum_{m=0}^\infty \frac{(-1)^{m}}{m!\Gamma(m+3/2)}\left(\frac{x}{2}\right)^{m+1/2}\] is an eigenfunction on the half-ball $D_+$ which is not in $H^2(D_+)$.
\item  The assumption that $v\in H^2(P)$ is for example true in the case of convex polygons in $\mathbb{R}^2$ where we assume that the angle between two consecutive sides $F_i^n$ and $ F_j^d$ is less than $\pi/2$, see \cite{Dy79}.
\end{itemize}
\end{remark}
\begin{proof}
Using the Rellich identity for the mixed Neumann-Dirichlet problem we
have that
\[\mu\left(\sum_{i=1}^j\int_{F_i^n}\mathbf{n}\cdot x\, v^2\, \dS-2\right)=\sum_{i=1}^j\int_{F_i^n}\mathbf{n}\cdot x\left|\grad_{\partial P}v\right|^2\, \dS-\sum_{i=1}^k \int_{F_i^d}\mathbf{n}\cdot x\, v_n^2\dS.\]
Using the above computation that $\mathbf{n}\cdot x= \dist(\mathbf{0})_i$ on each face of the polytope we get
\begin{multline*}
\mu\left(\sum_{i=1}^j\dist(\mathbf{0})_i\int_{F_i^n} v^2\, \dS-2\right)=\sum_{i=1}^j\dist(\mathbf{0})_i\int_{F_i^n}\left|\grad_{\partial P}v\right|^2\, \dS\\
-\sum_{i=1}^k\dist(\mathbf{0})_{j+i} \int_{F_i^d} v_n^2\dS.
\end{multline*}
where $|\grad_{\partial P} v|^2 = |\grad v|^2 - v_n^2$.
By translating the formula to an arbitrary point $p$ we get
\begin{multline*}
\mu\left(\sum_{i=1}^j\dist(p)_i\int_{F_i^n} v^2\, \dS-2\right)=\sum_{i=1}^j\dist(p)_i\int_{F_i^n}\left|\grad_{\partial P}v\right|^2\, \dS\\
-\sum_{i=1}^k\dist(p)_{j+i} \int_{F_i^d} v_n^2\dS.
\end{multline*}
\end{proof}
\begin{example}
   In this example we will only have Dirichlet condition on one of the sides of the square. Let $P=[0,1]\times [0,1]$ with the sides $F_1^n=\{0\}\times [0,1]$, $F_2^n=[0,1]\times \{1\}$, $F_3^n=\{1\}\times [0,1]$, and $F_1^d=[0,1]\times \{0\}$. We will let $p=(1/2,1/2)$, which means that $\dist_i(p)=1/2$.
   \begin{figure}[h!]
      \centering
   \begin{tikzpicture}[scale =3]
      \draw (0,0) node[below] {$(0,0)$} -- (1,0)  node[below] {$(1,0)$} -- (1,1) node[above] {$(1,1)$} -- (0,1)  node[above] {$(0,1)$}-- cycle;
      \fill (0.5,0.5) circle [radius=0.01];
      \node[below] at (0.5,0.499) {$p=(1/2,1/2)$};
      \node[below] at (0.5,0) {$F_1^d$};
      \node[left] at (0,0.5) {$F_1^n$};
      \node[above] at (0.5,1) {$F_2^n$};
      \node[right] at (1,0.5) {$F_3^n$};
   \end{tikzpicture}
   \end{figure}
   In this case the eigenfunctions are given by 
   \[v(x,y)=2\cos\left(\pi mx\right)\sin\left(\pi (n+1/2)y\right)\]
   with corresponding eigenvalues $\mu=\pi^2(m^2+(1/2+ n)^2)$ where $m, n\in \mathbb{N}\cup \{0\}$. 
   Computing the tangential gradient, we get
      \begin{multline*}
      \int_{F_1^n}|\grad_{\partial P}v|^2\dS=\int_{F_2^n}|\grad_{\partial P}v|^2\dS\\
      =4\int_0^1|\pi (n+1/2)\cos\left(\pi (n+1/2)y\right)|^2\mathrm{dy}=2\pi^2 (n+1/2)^2,
      \end{multline*}
      and 
      \[\int_{F_3^n}|\grad_{\partial P}v|^2\dS=2\pi^2m^2.\]
      The integral of the normal derivative is
      \[\int_{F_1^d}v_n^2\dS=4\int_0^1|\pi (n+1/2)\cos\left(\pi mx\right)|^2\mathrm{dx}=2\pi^2 (n+1/2)^2.\]
      In this case we have that 
      \begin{multline*}
      \dist(p)_1\int_{F_1^n}|\grad_{\partial P}v|^2\dS+\dist(p)_2\int_{F_2^n}|\grad_{\partial P}v|^2\dS\\
      +\dist(p)_3\int_{F_3^n}|\grad_{\partial P}v|^2\dS =2\pi^2 (n+1/2)^2+\pi^2m^2,
      \end{multline*}
      and 
      \[\dist(p)_{4}\int_{F_1^d}v_n^2\dS=\pi^2( n+1/2)^2.\]
      This means that 
    \[\sum_{i=1}^3\dist(p)_i\int_{F_i^n}\left|\grad_{\partial P}v\right|^2\, \dS-\dist(p)_{4} \int_{F_1^d} v_n^2\dS=\mu.\]
    We have also have that
     \[\int_{F_i^n}v^2\dS=2,\]
     and
     \[\dist(p)_1\int_{F_1^n}v^2\dS+\dist(p)_2\int_{F_2^n}v^2\dS+\dist(p)_3\int_{F_3^n}v^2\dS-2=1.\]
   Hence we see that \eqref{eq:chris:1} holds.
\end{example}
\begin{example}[Example \ref{ex:stekolov:square} continued]\label{ex:square}
We saw in Example \ref{ex:stekolov:square} that the solutions to the problem \eqref{eq:mixed_square} are given by
\[v(x,y)=2\cos\left(\pi mx\right)\sin\left(\pi ny\right)\]
      with corresponding eigenvalues $\mu=(\pi m)^2+(\pi n)^2$, where $m\in \mathbb{N}\cup \{0\}$ and $n\in \mathbb{N}\setminus \{0\}$. In this case we have that 
      $F_1^n=\{0\}\times [0,1]$, $F_2^n=\{1\}\times [0,1]$, $F_1^d=[0,1]\times \{0\}$ and $F_2^d=[0,1]\times \{1\}$. In this case we will let $p$ be arbitrary.
      On \[\int_{F_1^n}|\grad_{\partial P}v|^2\dS=\int_{F_2^n}|\grad_{\partial P}v|^2\dS=2(\pi n)^2\]
      and 
      \[\int_{F_1^d}v_n^2\dS=\int_{F_2^d}v_n^2\dS=2(\pi n)^2.\]
      In this case we have that 
      \[\dist(p)_1\int_{F_1^n}|\grad_{\partial P}v|^2\dS+\dist(p)_2\int_{F_2^n}|\grad_{\partial P}v|^2\dS=2(\pi n)^2,\]
      and 
      \[\dist(p)_{3}\int_{F_1^d}v_n^2\dS+\dist(p)_{4}\int_{F_2^d}v_n^2\dS=2(\pi n)^2.\]
      This means that 
    \[\sum_{i=1}^2\dist(p)_i\int_{F_i^n}\left|\grad_{\partial P}v\right|^2\, \dS-\sum_{i=1}^2\dist(p)_{i+2} \int_{F_i^d} v_n^2\dS=0.\]
     We have also have that
      \[\int_{F_1^n}v^2\dS=\int_{F_2^n}v^2\dS=2,\]
      and
      \[\dist(p)_1\int_{F_1^n}v^2\dS+\dist(p)_2\int_{F_2^n}v^2\dS=2.\]
    Hence we end up with zero on both sides in \eqref{eq:chris:1}.

    In general, one can change the point $p$ to change the right hand side and left hand side of \eqref{eq:chris:1}. However, in this case we get the same result independent of the point $p$.
\end{example}

\bibliographystyle{abbrv}
\bibliography{Bibliography}
\end{document}